\documentclass[11pt,centertags,oneside]{amsart}
\usepackage{amsmath,amstext,amsthm,amscd,typearea,hyperref}
\usepackage[a4paper,width=16cm,top=3cm,bottom=3cm]{geometry}
\usepackage{amssymb}
\usepackage{a4wide}
\usepackage[mathscr]{eucal}
\usepackage{mathrsfs}
\usepackage{typearea}
\usepackage{charter}
\usepackage[english]{babel}
\usepackage[utf8]{inputenc}  
\usepackage[T1]{fontenc}   
\DeclareUnicodeCharacter{00A0}{~}
\usepackage{enumerate}

\usepackage[all]{xy}
\usepackage{hyperref}
\usepackage{ amssymb }

\usepackage{eurosym}
\usepackage{psvectorian}
\usepackage{enumitem}
\setlist[itemize]{label=$\diamond$}
\usepackage{graphicx}  
\usepackage{xcolor,soul}
\usepackage{tikz-cd}

\newcommand{\R}{\mathbb{R}}  
 
\newcommand{\N}{\mathbb{N}} 
\newcommand{\Z}{\mathbb{Z}} 
\newcommand{\C}{\mathbb{C}} 
\newcommand{\D}{\mathbb{D}}

\newcommand{\tbif}{T_\mathrm{bif}}

\newcommand{\im}{\mathrm{Im}\,}
\newcommand{\md}{\mathcal{M}_d}

\newcommand{\la}{\lambda}

\newtheorem{theo}{Theorem}
\newtheorem{prop}[theo]{Proposition}
\newtheorem{coro}[theo]{Corollary}
\newtheorem{cor}[theo]{Corollary}

\newtheorem{defi}[theo]{Definition}
\newtheorem{rem}[theo]{Remark}
\newtheorem{lem}[theo]{Lemma}

\numberwithin{theo}{section}
\newcommand{\co}[1]{^{\circ {#1}}}

\newcommand{\limn}{\lim_{n \rightarrow \infty}}

\newcommand{\lam}{\lambda}

\newcommand{\mubif}{\mu_{\mathrm{bif}}}

\newcommand{\bif}{\mathrm{Bif}}

\newcommand{\supp}{\mathrm{supp}\,}
\newcommand{\re}{\mathrm{Re}\,}

\newcommand{\eps}{\epsilon}

\def\C{{\mathbb C}}
\def\N{{\mathbb N}}
\def\R{{\mathbb R}}
\def\Z{{\mathbb Z}}

\def\Q{{\mathbb Q}}

\def\D{\mathbb{D}}

\def\poly{\textup{Poly}}

\definecolor{Violet}{cmyk}{0.79,0.88,0,0}
\definecolor{ForestGreen}{cmyk}{0.76, 0, 0.76, 0.45}
\definecolor{Lavender}{cmyk}{0,0.48,0,0}

\title{Cubic Siegel polynomials and the bifurcation measure}
\author{Matthieu Astorg}
\email{matthieu.astorg@univ-orleans.fr}
\address{Institut Denis Poisson, Collegium Sciences et Techniques, Université d'Orléans, 
Rue de Chartres B.P. 6759, 45067 Orléans cedex 2 France.}

\author{Davoud Cheraghi}
\email{d.cheraghi@imperial.ac.uk}
\address{Department of Mathematics, Imperial College London, London SW7 2AZ, United Kingdom}

\author{Arnaud Chéritat}
\email{arnaud.cheritat@math.univ-toulouse.fr}
\address{Institut de Mathématiques de Toulouse,
	118 Route de Narbonne - Bâtiment 1R3
	31062 TOULOUSE CEDEX 9}

\date{\today}

\begin{document}

\begin{abstract}
We prove that cubic polynomial maps with a fixed Siegel disk and a critical orbit eventually landing 
inside { that} Siegel disk lie in the support of the bifurcation measure $\mubif$.  
{ This answers a question of Dujardin in positive}. 
{ Our result implies the existence of holomorphic disks in the support of $\mubif$, and also 
implies that} the set of  rigid parameters is not closed in the moduli space of cubic polynomials.
\end{abstract}

\maketitle

\section{Introduction}
Let $\poly_d$ denote the space of monic centered polynomial maps of degree $d \geq 2$, and let 
$\md$ denote the moduli space of rational maps of degree $d \geq 2$ on the Riemann sphere.
In general, we let $M=\poly_d$ or $\md$. 
By definition, the \emph{bifurcation locus} $\bif \subset M$ is the complement of the set of 
$J$-stability, that is, $\bif$ is the set of parameters around which the Julia set does not locally move holomorphically.
The celebrated results of Mañé-Sad-Sullivan \cite{MSS83} and Lyubich \cite{Ly83} provide several other 
equivalent characterisations of the bifurcation locus, in terms of the behaviour of the critical orbits or in terms 
of periodic cycles changing from repelling to attracting or vice-versa.

In \cite{demarco2001dynamics}, DeMarco introduced a \emph{bifurcation current} $\tbif$, which is a closed 
positive current of bidegree $(1,1)$ whose support is equal to $\bif$. 
Bassanelli and Berteloot \cite{berteloot2007bifurcation} studied the higher degree currents 
$\tbif^k:=\tbif \wedge \ldots \wedge \tbif$, $1 \leq k \leq \dim M$, which detect higher codimensional 
bifurcation phenomena. 
When $k=\dim M$, $\tbif^k$ is called the \emph{bifurcation measure} and is denoted by $\mubif$. 
This is a finite positive measure, which is the Monge-Ampère measure of the plurisubharmonic potential 
$M \ni f \mapsto L(f)$, where $L(f)$ is the Lyapunov exponent of the unique measure of maximal entropy.
The support of $\mubif$ is also called the maximal bifurcation locus, in the sense that it 
detects maximal codimension bifurcation phenomena.

In recent years, the study of the bifurcation measure $\mubif$ and its support has attracted considerable attention.
For instance, by the works of Bassanelli-Berteloot \cite{berteloot2007bifurcation}, 
Buff-Epstein \cite{buff2009bifurcation} and Dujardin-Favre \cite{dujardin2008distribution}, it is known 
that the support of $\mubif$ coincides with the closure of the set of parameters with the maximal number of 
non-repelling cycles (namely $d$ in the case of $\poly_d$ or $2d-2$ in the case of $\md$), and also with the 
closure of the set of parameters for which all critical points are strictly preperiodic to repelling cycles.

For parameters in the support of $\mubif$, all critical points must be active. 
However, this is not a sufficient condition. For instance, the following example is due to Douady 
\cite[Example 6.13]{dujardin2008distribution}. 
The polynomial map $f(z):=z+z^2/2+z^3 \in \poly_3$ has one parabolic fixed point which attracts both critical points, 
so they are both active at $f$ in the family $\poly_3$. However, $f$ is \emph{parabolic attracting}, 
which means that any small perturbation of $f$ possesses either a parabolic or an attracting fixed point. 
It follows that $f \notin \supp\,\mubif$.  
A more subtle example is given in \cite{inou2020support} by Inou and Mukherjee, { where} they 
construct a real-analytic { family} $(f_t)_{t \in I}$ of \emph{parabolic repelling} cubic polynomials, 
for which both critical points are { attracted to} the parabolic fixed point, { but} 
none of the $f_t$ { is} in the support of $\mubif$.
Loosely speaking, these examples show that in general a certain independence { of} 
the critical orbits is required for a map $f \in M$ to be in the support of $\mubif$.

In this paper, we address the following question { stated} by Dujardin in \cite{dujardin}: 
Given a cubic polynomial map with a Siegel disk and a critical orbit which eventually lands inside the Siegel disk, 
does { that} cubic polynomial belong to the support of $\mubif$?
Note that the second critical orbit must accumulate on the boundary of the Siegel disk, so that both critical points 
are active in $\poly_3$ but are related through the Siegel disk. 
Such maps also have a one-parameter family of quasi-conformal deformations, similar to the examples of Douady 
and Inou-Mukherjee. { However, in our case the parameter family is complex.} 

We work with the family of cubic polynomials with a marked fixed point placed at $0$,
parametrised as 
\begin{equation}\label{eq:markedcubic}
f_{\lam,a}(z)=\lam z + az^2+z^3, \quad \lambda, a\in \mathbb{C}.
\end{equation}

Let $\lam=e^{2i\pi \theta}$, for an irrational number $\theta$.  
The map $f_{\la,a}$ is \emph{linearisable} near the origin if there is a { conformal change 
of coordinate $\phi: V \to \D(0,r)$ defined on a neighbourhood $V$ of the origin such that 
$\phi  \circ f_{\la,a} = \la \phi$.} 
The maximal domain of that linearisation is called the \emph{Siegel disk} of $f_{\lam, a}$. 
By the classic work of { Siegel \cite{Sie42} and Brjuno \cite{Brj71}}, 
if $\theta$ satisfies the so-called \emph{Brjuno condition} 
$\sum_{n \geq 0} q_n^{-1} \log q_{n+1} <+\infty,$
then for every $a \in \mathbb{C}$, $f_{\lam, a}$ is linearisable near the origin. 
{ Here $(p_n/ q_n)_{n \in \N}$ is the sequence of best rational approximants of $\theta$.}

{ Let us say that} a cubic polynomial $f_{\la,a}$ is a \emph{Siegel polynomial of capture type}, 
if $f_{\la,a}$ has a Siegel disk containing the origin and a critical { point whose orbit 
eventually lands} inside that Siegel disk.
Our main theorem is the following:

\begin{theo}\label{th:supp}
{ Any Siegel polynomial of capture type $f_{\lam,a}$ 
is contained in the support of $\mubif$.}
\end{theo}

{ For a fixed $\lam= e^{2i \pi \theta}$}, it is not difficult to prove that the bifurcation locus 
of the slice $\{f_{\la,a} : a \in \C\}$ is contained in the support of $\mubif$. 
However, if $\theta$ is a Brjuno number and $f_{e^{2i\pi\theta},a}$ is a Siegel polynomial of  capture type, 
then it belongs to the stability locus of this slice. 
This implies { the following two} consequences:

\begin{coro}
There are holomorphic disks in the support of $\mubif$ in $\poly_3$.
\end{coro}

{Recall that a polynomial $P$ is called conformally rigid if any polynomial which 
is quasi-conformally conjugate to $P$ is affinely conjugate to $P$. 
Following the pioneering works of Sullivan \cite{Su85} and Thurston \cite{DH93}, now there is substantial body of work related to the 
quasi-conformal and conformal rigidity conjectures in complex dynamics (MLC conjecture and higher degree analogues).}

\begin{coro}
{ The set of conformally rigid cubic polynomials is not closed in $\poly_3$.}
\end{coro}

{ Indeed, Siegel polynomials of capture type are not conformally rigid (see Section \ref{sec:qcdef}), 
but strictly post-critically finite maps are conformally rigid \cite{DH93}, and form a dense subset of $\supp \mubif$ by \cite{buff2009bifurcation}.}

\subsection*{Acknowledgments}

The first named author has received funding from the ANR Project 
ANR PADAWAN /ANR-21-CE40-0012-01 and the  PHC Galileo project G24-123.
 The first and second named authors would like to thank CNRS-Imperial College London “Abraham de
Moivre” International Research Laboratory for their kind support and hospitality during the visit of M.~A. to Imperial College in Spring 2024, 
when this research was carried out.

\section{From elliptic to parabolic}\label{S:parabolic-elliptic}
Within Section~\ref{S:parabolic-elliptic} we assume that $\theta \in \R \setminus \Q$ is a bounded type number. 
In particular, $\theta$ is a Brjuno number. 

{ Let us fix $\la = e^{2i \pi \theta}$ and consider the family of maps
\[f_{\lam,a}(z)=\lam z + az^2+z^3, \quad  a\in \mathbb{C}.\]
We shall denote the Siegel disk of $f_{\la, a}$ with the notation $\Delta_\theta(a)$.} 

\subsection{Capture components}\label{sec:qcdef}
\begin{defi}
By a capture component we mean a connected component of the set 
\[\{a \in \C : f_{\la,a} \text{ is a polynomial of capture type }  \}.\]
\end{defi}

By \cite[Theorem 7.3]{Za99}, capture components are (open) stability components in the slice of maps $\{f_{\la,a}: a \in \C\}$, 
and are contained in the connectedness locus.
In particular, they are simply connected (by the maximum principle).
By \cite[Lemma 7.4]{Za99}, each capture component $V$ contains a unique \emph{centre}, that is, a parameter $a_0 \in V$ where a critical point of $f_{\la,a_0}$ 
is pre-periodic to $0$. 
Moreover, all cubic polynomials in a capture component minus its center point are quasi-conformally conjugate to each other on 
the Riemann sphere (\cite[Theorem 7.5 (a)]{Za99}).

\subsection{Asymptotic size and conformal radius}\label{SS:asymp-size}
Let $p_n/q_n$, for $n \geq 1$, be the sequence of the best rational approximants of $\theta$. 
Define $\la_n= e^{2i \pi p_n/q_n}$.

By \cite{Cheri2020}, for all $n \in \N^*$ there is a degree $q_n$ polynomial map $b_n$ such that
\begin{equation}\label{E:b_n-introduced}
f_{\lam_n,a}\co{q_n}(z)=z+b_n(a)z^{q_n+1} + O(z^{q_n+2}).
\end{equation}

{ Let us say that an open set $\Omega \subset \C$ is a non-degenerate parabolic locus, if for all sufficiently large $n \geq 1$, 
$b_n$ is non-zero on $\Omega$.}

Let $\mathcal{C}_\theta$ denote the Zakeri curve \cite{Za99}, i.e. 
\[\mathcal{C}_\theta = \{a \in \C : \text{ both critical points of $f_{\lambda, a}$ belong to $\partial \Delta_\theta(a)$}\}.\] 
Let $r_\theta(a)$ denote the conformal radius of the Siegel disk of $f_{\la,a}$ at $0$.

The following result will be crucial to our argument:

\begin{theo}[\cite{Cheri2020}]\label{th:bn}
For every irrational number of bounded type $\theta$ { and every $a \in \mathbb{C}$}, 
\[\lim_{n\to \infty} |b_n(a)|^{1/q_n} = 1/r_\theta(a).\] 
{ Moreover, the convergence is uniform on compact subsets of any non-degenerate parabolic locus in $\C$.} 
\end{theo}

{ The first part of the above theorem is readily presented in \cite{Cheri2020}. 
The latter part also follows from that paper. 
Indeed, if $\Omega \subset \C$ is a non-degenerate parabolic locus, for large $n$, the sequence of maps $q_n^{-1} \log |b_n(a)|$ are harmonic on $\Omega$, 
and by \cite[Proposition 54]{Cheri2020}, converge to $-\log r_\theta(a)$ in $L^1_{loc}(\Omega)$. 
Then, the $L^1_{loc}$ convergence implies uniform convergence on compact sets for harmonic functions.}

\subsection{Parabolic rays converge to Siegel internal rays}
Let $U \subset \C$ be an arbitrary capture component. There exists a critical point $c_a$ and an integer 
$k \geq 1$ such that for all $a \in U$, $f_{{ \lam},a}^{\circ k}(c_a) \in \Delta_\theta(a)$. 
We let $U^*:=\{a \in U : f_{{ \lam},a}^{\circ k}(c_a) \neq 0  \}$. 
For $a \in U$, let $\phi_a : \Delta_\theta(a) \to \D_{r_\theta(a)}$ denote the linearising coordinate, 
normalised by $\phi_a'(0)=1$. 
{ We make use of the following lemma of Jellouli \cite{Jellouli94}, which is obtained through 
direct calculations, a compactness argument, and the bound $|\theta_0-p_n/q_n| \leq 1/q_n^2$.} 

\begin{lem}[Jellouli]\label{lem:jell}
{ For every $r_0>0$, there exists a constant $C>0$ such that for all $a \in U$, all $z \in \D(0,r_0r_\theta(a))$, all $n\geq 1$, and all 
$0 \leq k \leq q_n$:}
\[|\phi_a \circ f_{\la_n,a}^{\circ k} \circ \phi_a^{-1}(z) - \la_n^k z | \leq C k|z|/q_n^2.\]
\end{lem}

There is a holomorphic motion of the boundary of $\Delta_\theta(a)$ over $U$. 
The map $(a,z) \mapsto \phi_a(z)$ is holomorphic over $a \in U$ and $z \in \Delta_\theta(a)$, 
and the map $a \mapsto \log r_\theta(a)$ is harmonic on $U$,
{ see \cite{Sullivan83} or \cite{Zakeri16}}. 
In particular, there exists a holomorphic function $u : U \to \C$ such that $\re u = \log r_\theta$.
We let $v:=e^u$, so that $v$ is also holomorphic and $|v|=r_\theta$.

For $a \in U$ and $n \geq 1$, consider the map 
\[\tilde G_{n,a}=\phi_a \circ f_{\lam_n,a}^{\circ q_n} \circ \phi_{a}^{-1},\] 
which is defined on a neighbourhood of $0$. 
We also consider the normalised map  
\[G_{n,a}(z) := v(a)^{-1} \tilde G_{n,a}(v(a) z).\]
From \eqref{E:b_n-introduced}, it is not difficult to see that near $0$, 
\begin{equation}\label{E:G-expansion}
G_{n,a}(z)=z + b_n(a) v(a)^{\circ q_n} z^{q_{n+1}} + O(z^{q_n+2}).
\end{equation}

For any $r_1 \in (0,1)$, there exists $N \in \N$ such that for all $a \in U$ and all $n \geq N$, $G_{n,a}$ is 
defined and univalent on $\D(0, r_1)$. 
In particular, we may write 
\[G_{n,a}(z) = z e^{g_{n,a}(z)}\] 
for some function $g_{n,a} : \D(0,{ r_1}) \to \C$. 
Then, the map 
\[H_{n,a}(w):=w + g_{n,a}(e^w)\] 
is defined and univalent on the left half-plane $\re w < \log r_1$, and satisfies 
\[\exp \circ H_{n,a} = G_{n,a} \circ \exp.\] 

{ We do not know if $U$ is a non-degenerate parabolic locus. 
Let $\Omega \subseteq U^*$ be a simply connected non-degenerate parabolic locus. In particular, the convergence in Theroem~\ref{th:bn} 
is uniform on compact subsets of $\Omega$.} 

\begin{lem}\cite[Lem 4.2]{Cheri01}\label{lem:nofp0}
{ Fix an arbitrary $r_1 \in (0,1)$. For $n\geq 1$ and $a\in U$} let $N(n,a)$ denote the number of fixed points of $G_{n,a}$ in $\D(0,r_1)$.
Then, as $n \to \infty$, $N(n,a)/q_n$ converges to $0$ { uniformly on compact subsets of $\Omega$}. 
\end{lem}

\begin{proof}
{ Let $\Lambda$ be a compact set in $\Omega$.} 
For each $n\geq 1$, $N(n,a)$ is uniformly bounded from above { independent} of $a \in \overline{U}$. 
Let us choose $a_n \in \Lambda$ so that $N(n,a_n)= \max_{a \in \Lambda} N(n,a)$.
{ By Lemma~\ref{lem:jell},} there exists a sequence $r_n \to 1^-$ such that for every $n \in \N$, 
$G_{n,a_n}$ is defined and univalent on $\D(0,r_n)$, and $\sup_{|z| \leq r_n} |G_{n,a_n}(z)-z|  \to 0$. 
We consider the normalised map $F_n(z):=r_n^{-1} G_{n,a_n}(r_n z)$, so that for every $n \in \N^*$, 
$F_n$ is defined and univalent on $\D$, and $\sup_{|z| \leq 1} |F_n(z)-z| \to 0$.
In particular, we will always assume in what follows that $n$ is large enough so that 
\[\sup_{|z| \leq 1} |F_n(z)-z| \leq 2.\]
Moreover, we have
\[F_n(z) = z + C_n z^{q_n+1} + O(z^{q_n+2})\]
where $C_n:=b_n(a_n)v(a_n)^{q_n} r_n^{q_n}$. 
By Theorem \ref{th:bn} and the definition of $v$, we have
\begin{equation}\label{E:C_n-size}
\lim_{n \to + \infty} \frac{1}{q_n} \log |C_n| = 0.
\end{equation}
	
For each { $r \in (0,1]$,} let $N_n(r)$ denote the number of fixed points of $F_n$ in $\D(0,r) \setminus \{0\}$, 
counted with multiplicity. 
Let $\zeta_{n,j}$, for $1 \leq j \leq N_n(1)$, denote the fixed points of $F_n$ in $\D \setminus \{0\}$, repeated 
with multiplicity.
There exists a holomorphic function { $\sigma_n: \D \to \C$} with no zeroes in $\D$ such that for all $z \in \D$, 
\[F_n(z) = z + z^{q_n+1} { \sigma_n(z)} \prod_{j=1}^{N_n(1)} \frac{z - \zeta_{j,n}}{1 - \overline{\zeta_{j,n}}z}.\]
Clearly $\max_{|z|=1}|{ \sigma_n(z)}| \leq \max_{|z|=1} |F_n(z)-z| \leq 2$, thus by the maximum principle, $|{ \sigma_n(0)}| \leq 2$.   
Moreover, $C_n={ \sigma_n(0)} (-1)^{N_n(1)} \prod_{j=1}^{N_n(1)} \zeta_{j,n}$.
Let $r \in (r_1, 1 )$. 
Then
\[|C_n|= |{ \sigma}_n(0)| \prod_{|\zeta_{j,n}|<r } |\zeta_{j,n}| \prod_{|\zeta_{j,n}|\geq r } |\zeta_{j,n}| 
\leq 2 \cdot r^{N_n(r)} \cdot 1.\] 
Therefore, 
\[0 \leq \frac{N_n(r)}{q_n} \leq \frac{\log 2-\log |C_n| }{q_n|\log r|}.\]
{ Combining with \eqref{E:C_n-size}, we get} 
\[\limn \frac{N_n(r)}{q_n}=0.\]
Since $N_n(r) \geq N(n,a_n)$ for large enough $n$, we are done.
\end{proof}

\begin{lem}\label{L:no-fixed-points}
{ Let $r_1 \in (0,1)$, and $\Lambda$ be a compact set in $\Omega$.} 
There exists $N_1 \in \N$ such that for all $n \geq N_1$ and all $a \in \Lambda$, $G_{n,a}$ has no fixed points 
in $\D(0,r_1)$ except for $z=0$.
\end{lem}

\begin{proof}
By Lemma \ref{lem:jell}, there exists $C=C(r_1)>0$ such that { for all $n\geq 1$,} all $1\leq k \leq q_n$, all { $a \in \Lambda$} 
and all $|z| \leq r_1 r_\theta(a)$, we have 
\[| \phi_a \circ f_{\lam_n,a}^{\circ k} \circ \phi_a^{-1}(z)  { -\lam_n^k z} | \leq C k |z|/q_n^2.\]
Let us consider $L_{n,a}(z):=v(a)^{-1}\phi_a \circ f_{\lam_n,a} \circ \phi_a^{-1}\left( z v(a) \right)$. 
We have $L_{n,a}^{\circ q_n}=G_{n,a}$, and for all $n\geq 1$, all $1 \leq k \leq q_n$ and all $z \in \D(0,r_1)$, 
\[|L_{n,a}^{\circ k}(z) - \la_n^k z | \leq C k |z|/q_n^2.\]
Let $s:=\min(1, 1/C)>0$. 
By Lemma \ref{lem:nofp0}, there exists $N_1 \in \N$ such that for all $n \geq N_1$ and all { $a \in \Lambda$},
\begin{equation}\label{E:bound-by-ratio}
N(n,a) < s q_n/2.
\end{equation}
Assume for a contradiction that there exists $n \geq N_1$, { $a \in \Lambda$} and $z \in \D(0,r_1) \setminus \{0\}$ such that 
$G_{n,a}(z)=z$. 
Then for all $1 \leq k \leq q_n$, $L_{n,a}^{\circ k}(z)$ is also a fixed point of $G_{n,a}$.
{ We will show that the set $\{ L_{n,a}^{\circ k}(z)  : 1 \leq k \leq q_n  \}$ contains at least $sq_n$ elements, 
which contradicts \eqref{E:bound-by-ratio}.} 
	
First, observe that if $1 \leq k_1< k_2 \leq q_n$, then 
\[| \la_n^{k_1} - \la_n^{k_2}| \geq  4/q_n.\]
Then, for all $1 \leq k_1< k_2 \leq s q_n$:
\begin{align*}
|  L_{n,a}^{\circ k_1}(z) -  L_{n,a}^{\circ k_2}(z) | & \geq | \la_n^{k_1}z - \la_n^{k_2}z| 
- |L_{n,a}^{ \circ k_1}(z) - \la_n^{k_1} z | - |L_{n,a}^{\circ k_2}(z) - \la_n^{k_2} z | \\
&\geq \frac{4}{q_n} |z| - \frac{2C s q_n |z|}{q_n^2} \\
&\geq \frac{2|z|}{q_n}.
\end{align*}
In particular, $L_{n,a}^{\circ k_1}(z) \neq  L_{n,a}^{\circ k_2}(z)$. 
{ This completes the proof of the claim.}
\end{proof}

{ In the following lemma we assume that $r_1 \in (0,1)$, { $\Lambda \subset \Omega$} and $N_1 \in \mathbb{N}$ satisfy 
Lemma~\ref{L:no-fixed-points}.} 

\begin{lem}\label{lem:nofp}
{ For all $r_1 \in (0,1)$, all $n\geq N_1$ and all { $a \in \Lambda$}, on the left half-plane $\re w < \log r_1$}, we may write 	
\[g_{n,a}(e^w) = e^{\ell_{n,a}(w)}\]
where 
\[\ell_{n,a}(w) = \log b_n(a)+ q_n u(a)+ q_n w + k_{n,a}(e^w)\]
{ for some holomorphic function $k_{n,a}: \mathbb{D}(0,r_1) \to \mathbb{C}$ satisfying $k_{n,a}(0)=0$.} 
\end{lem}

Note that $\log b_n(a)$ is well-defined on $\Omega$ for large enough $n$, { because $\Omega$ is simply connected and $b_n(a)$ 
becomes non-zero on $\Omega$.} 

\begin{proof}
Recall that by definition, we have $g_{n,a}(z)=\log \frac{G_{n,a}(z)}{z}$, well-defined 
on $\D(0,r_1)$ for all $n \geq N_1$ and { $a \in \Lambda$}. 
Moreover, { by \eqref{E:G-expansion} and the definition of $v$, we have} 
\[G_{n,a}(z)=z+ b_n (a) e^{q_n u(a)} z^{q_n +1} + O(z^{q_n+2}).\] 
Therefore, 
\begin{equation}\label{E:g_n-expansion}
g_{n,a}(z) = \log \left( 1 + b_n(a) e^{q_n u(a)} z^{q_n} + O(z^{q_n+1})\right) 
= b_n(a) e^{q_n u(a)} z^{q_n} h_{n,a}(z)
\end{equation}
for some function $h_{n,a}: \D(0,{ r_1}) \to \C$ such that $h_{n,a}(0)=1$. 
Moreover, because by Lemma~\ref{lem:nofp}, $G_{n,a}$ has no fixed points inside $\D(0,{ r_1}) \setminus \{0\}$, 
$h_{n,a}$ does not vanish on $\D(0,{ r_1})$. 
Thus, $h_{n,a}(z)=\exp \circ k_{n,a}(z)$ for some function $k_{n,a} : \D(0,{ r_1}) \to \mathbb{C}$ 
with $k_{n,a}(0)=0$. 
\end{proof}

\begin{lem}\label{L:k_n-behaviour}
{ For every $r_0 \in (0,1)$ and every compact set $\Lambda \subset \Omega$}, we have 
\[\limn \sup _{a \in { \Lambda}} \sup_{|z| \leq r_0} \frac{\re k_{n,a}(z)}{q_n} = 0, \qquad \text{and} \qquad 
 \limn \sup _{a \in { \Lambda}} \sup_{|z| \leq r_0} \left | \frac{k_{n,a}'(z)}{q_n} \right |= 0.\]   
\end{lem}

\begin{proof}
Let us fix arbitrary $r_0 \in (0,1)$, { compact set $\Lambda \subset \Omega$} and $\eps>0$. 
We choose $r_1 \in (r_0,1)$ such that $- \log r_1 < \eps/2$.
By Lemma \ref{lem:jell}, if we let
\[\eps_n:= \sup_{a \in U} \sup_{|z|\leq r_1} |G_{n,a}(z)-z|, \]
then we have $\eps_n \leq C(r_1)/q_n$, and hence $\eps_n \to 0$. 
	
It follows { from \eqref{E:g_n-expansion}} that for all $(z,a) \in \D(0,r_1) \times { \Lambda}$, we have 
\[|b_n(a) e^{q_n u(a)} z^{q_n+1} e^{k_{n,a}(z)}  | \leq \eps_n.\]
Let $(y_n,a_n) \in \overline{\D}(0,r_1) \times { \Lambda}$ be such that
\[| e^{k_{n,a_n}(y_n)}  |  = \max_{|z| \leq r_0, a \in { \Lambda}} | e^{k_{n,a}(z)}  |.\]
By the maximum principle, we have $|y_n|=r_1$ and $a_n \in \partial U$. Therefore, 
\begin{align*}
r_\theta(a_n)^{q_n} |b_n(a_n)| r_1^{{ q_n}} e^{\re k_{n,a_n}(y_n)}  & \leq \eps_n/r_1,
\end{align*}
using ${ |}e^{q_n u(a)} { |} = r_\theta(a)^{q_n}$.
Applying $\frac{1}{q_n} \log$ to both sides of the above inequality, we obtain 
\[{ \log }  r_\theta(a_n)+\frac{1}{q_n} \log |b_n(a_n)| + \log r_1 + \frac{\re k_{n,a_n}(y_n)}{q_n} \leq \frac{1}{q_n} \log  \frac{\eps_n}{r_1}\]
and hence 
\begin{align*}
\frac{\re k_{n,a_n}(y_n)}{q_n} 
&\leq \frac{1}{q_n} \log \frac{\eps_n}{r_1} - \log r_1 - \frac{1}{q_n} \log |b_n(a_n)| - \log r_\theta(a_n).
\end{align*}
{ Using $\eps_n \leq C(r_1)/q_n$,} we have $\limsup_{n \to + \infty} \frac{1}{q_n} \log \frac{\eps_n}{r_1} \leq 0$, 
and by Theorem \ref{th:bn}, 
\[\limn  \frac{1}{q_n} \log |b_n(a_n)| - \log r_\theta(a_n) = 0.\]
Therefore, there exists $N \in \N$ such that for all $n \geq N$, 
\[\frac{1}{q_n} \log |b_n(a_n)| - \log r_\theta(a_n)  + \frac{1}{q_n} \log \frac{\eps_n}{r_1} \leq \frac{\eps}{2}.\]
Then, by our choice of $r_1$, for all $n \geq N$,  
\[\frac{1}{q_n} \re k_{n,a_n}(y_n) \leq - \log r_1 + \frac{\eps}{2} \leq \eps.\] 

{ On the other hand, $k_{n,a}(0)=0$ for all $n \in \N$ and $a \in { \Lambda}$.}
Combining with the above argument, we conclude that for all $n\geq N$, 
\[0 \leq \sup _{a \in { \Lambda}} \sup_{|z| \leq r_0} \frac{\re k_{n,a}(z)}{q_n}  \leq   \frac{1}{q_n} \re k_{n,a_n}(y_n) \leq \eps\]

{ The first inequality implies the second inequality, as a general property. 
Fix arbitrary $r_0$, $r_1 \in (r_0, 1)$ and $\epsilon>0$. 
By the first inequality, there is $N \in \N$ such that for all $n \geq N$ all $a \in \Lambda$ and all $|z| < r_1$, 
$\re k_{n,a}(z)/(\epsilon q_n) < 1$. Thus, $k_{n, a}/(\epsilon q_n)$ maps the disk $\D(0, r_1)$ into the left 
half-plane $\re w < 1$, with $k_{n,a}(0)/(\epsilon q_n)=0$. 
Post-composing the map $k_{n,a}/(\epsilon q_n)$ with a M\"obius transformation and applying the Schwarz lemma, 
we conclude that there is a constant $C$, depending only on $r_1$ but not on $n$, $\epsilon$ and $a$, 
such that on the disk $\D(0, r_0)$, $|k'_{n,a}(z)/(\epsilon q_n)| \leq C$.
As $\epsilon$ was arbitrary, this implies the second inequality.} 	
\end{proof}

\begin{prop}\label{P:bands-in-petals}
{ For every $r_0 \in (0,1)$ and every compact set $\Lambda \subset \Omega$,} there exists $N \in \N$ such that for all $n \geq N$, 
all $a \in { \Lambda}$, { and all $k \in \mathbb{Z}$,} 
the regions 
\[C_{n,a}^k
=\{ w \in \C : \re w<\log r_0 \text{ and }  2k\pi+3\pi/4 < \im \ell_{n,a}(w) < 2k\pi + 5\pi/4\}\]
are forward invariant by $H_{n,a}$, and for every $w \in C_{n,a}^k$, $\re H_{n,k}^{\circ m} (w) \to -\infty$ as $m \to +\infty$. 
\end{prop}

\begin{proof}
{  
Fix an arbitrary $r_0 \in (0,1)$.
By Lemmas~\ref{lem:nofp} and \ref{L:k_n-behaviour}, there is $N \in \mathbb{N}$ such that for all $n\geq N$, all $a \in U$, and all 
$w$ in the left half-plane $\re w < \log r_0$, $\ell_{n,a}(w)$ is defined, and 
\begin{equation}\label{E:P:bands-in-petals-1}
\re k_{n,a}(e^w) < q_n \min \left \{ \frac{1}{4}, \frac{-\log r_0}{2} \right \} 
\qquad \text{and} \qquad  |k'_{n,a}(e^w)| < q_n \min \left \{ \frac{1}{4}, \frac{-\log r_0}{2} \right \}.
\end{equation}
Below we assume that $n\geq N$. 
Let us also fix $k \in \Z$ and $a \in U$. 

We divide the set $C_{n,a}^k$ into the sets 
\begin{align*} 
C_{n,a}^{k, t} &=\{ w \in \C : \re w<\log r_0 \text{ and }  2k\pi+ 7\pi/6 < \im \ell_{n,a}(w) < 2k\pi + 5\pi/4\}, \\
C_{n,a}^{k,c} & =\{ w \in \C : \re w < \log r_0 \text{ and }  2k\pi+5\pi/6 \leq \im \ell_{n,a}(w) \leq 2k\pi + 7\pi/6\}, \\
C_{n,a}^{k,b} & =\{ w \in \C : \re w < \log r_0 \text{ and }  2k\pi+3\pi/4 < \im \ell_{n,a}(w) < 2k\pi + 5\pi/6\}. 
\end{align*}
We also consider the curves  
\begin{align*}
\rho_{n,a}^{k,1} &= \{ w \in \C : \re w < \log r_0  \text{ and }  \im \ell_{n,a}(w)= 2k\pi+3\pi/4\}, \\
\rho_{n,a}^{k,2} &= \{ w \in \C : \re w < \log r_0  \text{ and }  \im \ell_{n,a}(w)= 2k\pi+5\pi/6\}, \\
\rho_{n,a}^{k,3} &= \{ w \in \C : \re w < \log r_0  \text{ and }  \im \ell_{n,a}(w)= 2k\pi+7\pi/6\}, \\
\rho_{n,a}^{k,4} &= \{ w \in \C : \re w < \log r_0  \text{ and }  \im \ell_{n,a}(w)= 2k\pi+5\pi/4\}.
\end{align*}
First we prove some properties of these sets, and the map $H_{n,a}$ on them. 

\medskip

\textbf{P1.} For $i= 1,3$, we have 
\[\inf  \big \{|w_1 -w_2| : w_1 \in \rho_{n,a}^{k,i}, w_2 \in \rho_{n,a}^{k,i+1} \big\} \geq \frac{1}{6q_n}.\]

\smallskip

Assume in the contrary that there are $i\in \{1,3\}$, $w_1 \in  \rho_{n,a}^{k,i}$ and $w_2 \in \rho_{n,a}^{k,i+1}$ 
such that $|w_1 -w_2| < 1/(6q_n)$. 
Then, by \eqref{E:P:bands-in-petals-1}, we must have 
\begin{align*}
\pi/12 & =\ell_{n,a}(w_2)- \ell_{n,a}(w_1) \\
& \leq |w_2 - w_1| \cdot  \sup \{|\ell_{n,a}'(w)|  : w\in \C, \re w < \log r_0\} \\
& < \frac{1}{6q_n} \cdot \left (q_n+ \sup\{ |k_{n,a}'(e^w) e^w| :  w\in \C, \re w < \log r_0\}\right) \\
& \leq \frac{1}{6q_n} \cdot \left (q_n+ \frac{q_n}{4} \right ), 
\end{align*}
which is a contradiction.  

\medskip

\textbf{P2.} For $i=1,2,3,4$, and all distinct $w_1, w_2 \in \rho_{n,a}^{k,i}$, 
\[\arg (w_2 - w_1) \in (-\pi/6, +\pi/6)+ \pi \mathbb{Z}.\]  

\smallskip 

For this, it is enough to note that for all $w \in \rho_{n,a}^{k,i}$ we have 
\begin{align*}
|\arg \ell'_{n,a}(w)| = | \arg (q_n+ k'_{n,a}(w) e^w )|  \leq \arcsin (1/4)< \pi/6. 
\end{align*}
Above, we have used $|e^w| <1$ and \eqref{E:P:bands-in-petals-1}. 
In particular, $\ell_{n,k}'(w) \neq 0$. 

\medskip

\textbf{P3.} We have 
\[\arg (H_{n,a}(w)-w) \in (3\pi/4, 5\pi/6)+ 2\pi \mathbb{Z}, \quad \text{for all } w \in C_{n,a}^{k,b}\] 
and 
\[\arg (H_{n,a}(w)-w) \in (7\pi/6, 5\pi/4) + 2 \pi \mathbb{Z} \quad \text{ for all } w \in C_{n,a}^{k,t}.\]

\smallskip 

By Lemma~\ref{lem:nofp}, for all $w$ in the left half-plane $\re w< \log r_0$, 
\begin{align*}
H_{n,a}(w) = w + g_{n,a}(e^{w}) = w+ e^{\ell_{n,a}(w)}
\end{align*}
Then,
\begin{align*}
\arg (H_{n,a}(w)- w)= \arg e^{\ell_{n,a}(w)}= \im \ell_{n,a}(w).
\end{align*}
Combining with the definitions of $C_{n,a}^{k,b}$ and $C_{n,a}^{k,t}$ we obtain P3. 

\medskip

\textbf{P4.} For all $w \in C_{n,a}^{k,c}$, 
\[\arg (H_{n,a}(w)-w) \in (5\pi/6, 7\pi/6)+ 2\pi \mathbb{Z}.\] 

\smallskip

We presented P4 separately, for the sake of the clarity of the later arguments. Otherwise, the proof is already given in P3.

\medskip

\textbf{P5.} There is $N_2 \in \N$, independent of $a \in \Lambda$ and $k\in \Z$, such that for all $n \geq N_2$ and all $w \in C_{n,a}^k$, 
we have 
\[|H_{n,a}(w)-w| \leq 1/(6q_n).\] 

\smallskip 

To see this, note that by our choice of $N$ for \eqref{E:P:bands-in-petals-1}, 
\begin{align*}
\left |H_{n,a}(w) - w \right| & = |e^{\ell_{n,a}(w)}| \\
&  \leq |b_n(a)| |e^{q_n u(a)}| e^{q_n \log r_0 -q_n (\log r_0)/2} \\
& = |b_n(a)| r_\theta(a)^{q_n}  e^{q_n(\log r_0)/2} 
\end{align*}
By Theorem \ref{th:bn}, $\lim_{n \to \infty} |b_n(a)|^{1/q_n} r_\theta(a)=1$. 
Thus, for large enough $n$, 
\[|b_n(a)| r_\theta(a)^{q_n} \leq e^{q_n(-\log r_0)/4}.\]  
Combining with the previous equation, we conclude that for large enough $n$, 
\begin{align*}
\left |H_{n,a}(w) - w \right| & \leq e^{q_n(\log r_0)/4} < 1/(6q_n).
\end{align*}

Now we are ready to complete the proof. By P2, every $\rho_{n,a}^{k,i}$, for $i=1,2,3,4$, meets the 
vertical line $\re w=\log r_0$ at a single point, and divides the left half-plane $\re w < \log r_0$ into two connected components. 
Thus we may talk about the component below or above $\rho_{n,a}^{k,i}$ in that left half-plane. 
Let us first show that $H_{n,a}(C_{n,a}^{k,b}) \subset C_{n,a}^k$. 
Fix an arbitrary $w \in C_{n,a}^{k,b}$. 
Because $\re w < \log r_0$, by P3, $\re H_{n,a}(w) < \log r_0$. 
By P3 and P2, $H_{n,a}(w)$ lies above the curve $\rho_{n,a}^{k,1}$. 
By P5 and P1, $H_{n,a}(w)$ lies below the curve $\rho_{n,a}^{k,4}$.
Combining these, we conclude that $H_{n,a}(w) \in C_{n,a}^k$. 
By a symmetric argument, $H_{n,a}(C_{n,a}^{k,t}) \subset C_{n,a}^k$. 
On the other hand, by P5 and P1, $H_{n,a}(C_{n,a}^{k,c})$ is contained in $C_{n,a}^k$. 
Therefore, $C_{n,a}^k$ is invariant under $H_{n,a}$. 

By P3 and P4, for every $w$ in the closure of $C_{n,a}^k$, $\re H_{n,a}(w) < \re w$. 
This implies that for every $w\in C_{n,a}^k$, $\re H_{n,a}^{\circ m}(w) \to -\infty$, as $m \to +\infty$.
}
\end{proof}

\begin{cor}\label{C:petal-in-basin}
For every $k_1, k_2 \in \Z$, $\exp(C_{n, a}^{k_i})$ is contained in a connected component { $\mathcal{P}_{n,a}^{k_i}$} 
of the immediate parabolic basin of $f_{\lam_n,a}$. Moreover, if $k_1 - k_2  \notin  q_n \Z$, 
then { $\mathcal{P}_{n,a}^{k_1} \neq \mathcal{P}_{n,a}^{k_2}$}.
\end{cor}

\begin{proof}
{ By Proposition~\ref{P:bands-in-petals}, and the definitions of $H_{n,a}$, $G_{n,a}$ and $f_{\lambda_n,a}$, 
the sets $\exp(C_{n,a}^{k})$ are contained in the immediate parabolic basins of $0$ for $f_{\la_n,a}$. 
When $k_1 - k_2  \notin  q_n 2\pi\Z$, $\exp(C_{n,a}^{k_1})$ and $\exp(C_{n,a}^{k_2})$ land at $0$ at well-defined angles, 
where the angel between them is a non-zero integer multiple of $2\pi/q_n$. 
On the other hand, since $b_n(a) \neq 0$, there are exactly $q_n$ repelling petals, 
landing at 0 at well-defined angles, with equally spaced at angle $2\pi/q_n$ between consecutive ones. 
This implies that for $k_1 - k_2  \notin q_n 2\pi \Z$, $\exp(C_{n,a}^{k_1})$ and $\exp(C_{n,a}^{k_2})$ are disjoint. 
}
\end{proof}

\section{Proof of Theorem \ref{th:supp}}
{ There are two cases to consider, depending on whether a parameter $a$ in a capture component $U$ lies in a non-degenerate parabolic locus, or not. 
The former case is the most difficult to deal with, and we only look at the later case near the end of this section. 
Let us continue to use the notations introduced in the previous sections. 
In particular, $\Omega$ is a non-degenerate parabolic locus in $U^*$ and $\Lambda \subset \Omega$ is an arbitrary compact set.
From now on we also further assume that $\Omega$ is compactly contained in $U^*$. 
 
Recall that $c_{\la,a}$ is a critical point of $f_{\la,a}$ which is mapped into the Siegel disk $\Delta_{\theta}(a)$ in $k$ iterates. 
Consider the functions 
\[K(a):=\log \phi_a \circ f_{\la,a}^{\circ k}(c_{\la,a}),\] 
and $K_n: \Omega \to \C$,  where 
\[K_n(a):=\log \phi_a \circ f_{\la_n,a}^{\circ k}(c_{\la_n,a}).\]
Because $\Omega$ is compactly contained in $U^*$, for large enough $n$, the corresponding critical point $c_{\la_n,a}$ of $f_{\la_n,a}$ 
will be mapped into $\Delta_{\theta}(a) \setminus \{0\}$ in $k$ iterates as well.
Also, because $\Omega$ is chosen simply connected, the $\log$ function is well-defined.
For convenience, let us assume that $K_n$ is defined for all $n$ (otherwise one considers all sufficiently large $n$).}
Clearly, $K_n$ uniformly converges to $K$ on $\Omega$. 

\begin{lem}\label{L:chi_n-convergence}
{ For every $a_0 \in \Omega$,} the sequence of functions 
\[\chi_n(a)=\frac{1}{q_n} \ell_{n,a} \circ K_n(a) -\frac{1}{q_n} \ell_{n,a_0} \circ K_{ n}(a_0)\] 
converges uniformly { on compact subsets of $\Omega$} to a non-constant holomorphic function { $\chi: \Omega \to \C$}. 
\end{lem}

\begin{proof}
{ As $a$ varies in a compact subset of $\Omega$, $\phi_a \circ f_{\la,a}^{\circ k}(c_{\la,a})$ forms a compact subset of $\D(0,1)$. Employing 
Lemma~\ref{L:k_n-behaviour}, we conclude that 
\[\frac{1}{q_n} \ell_{n,a} \circ K_n(a) = \frac{1}{q_n} \log b_n(a)+u(a) + K(a) + o(1),\]
on any compact subset of $\Omega$, with the constant in $o$ depending only on that compact set.} 
Thus, 
\[\re \frac{1}{q_n} \ell_{n,a} \circ K_n(a) = \frac{1}{q_n} \log |b_n(a)| +  \log r_\theta(a)+ \log |\phi_a \circ f_{\la,a}^{\circ k}(c_a)| + o(1).\]
By Theorem \ref{th:bn}, { on any compact subset of $\Omega$,} $\frac{1}{q_n} \log |b_n(a) | \to - \log r_\theta(a)$.
{ In particular, on any compact subset of $\Omega$, as $n \to \infty$,}  
\[\re \chi_n \to \re \left(K(a) - K(a_0) \right)= \log |\phi_a \circ f_{\la,a}^{\circ k}(c_a)| -  	\log |\phi_{a_0} \circ f_{\la,a_0}^{\circ k}(c_{a_0})|.\] 
By definition, $\chi_n(a_0)=0$. 
It follows that $\chi_n$ converges to $K(a) - K(a_0)$, uniformly on compact subsets of $\Omega$.
	
{ By a classical argument of quasi-conformal deformation, $K(a) - K(a_0)$ is not constant (see for instance \cite[Theorem 7.3]{Za99}).}
\end{proof}

\begin{prop}\label{P:perturb}
{ For every $a_0 \in \Omega$ and every $\eps>0$ there is $N \in \N$ satisfying the following property. 
For all $n\geq N$, there are $a, b \in B(a_0, \eps)$ and integers $k_a$ and $k_b$ with $k_a - k_b \notin q_n \Z$ such that 
\[K_n(a) \in C_{n, a}^{k_a} \qquad \text{ and } \qquad  K_n(b) \in C_{n,b}^{k_b}.\]}
\end{prop}

\begin{proof}
{ Fix arbitrary $a_0 \in \Omega$ and $\eps>0$. 
By making $\eps$ smaller, we may assume that the closure of $B(a_0, \eps)$ is contained in $\Omega$. 
Since $\Omega$ is compactly contained in $U^*$, $K(\Omega)$ is compactly contain the left half-plane 
$\re w <0$. As $K_n$ converges to $K$ uniformly on $\Omega$, there is $r_0 \in (0,1)$ such that for sufficiently large $n$, 
$K_n(\Omega)$ is contain in the left half-plane $\re w < \log r_0$.

It follows from Lemma~\ref{L:chi_n-convergence} that there is $\delta>0$ such that for large enough $n$, 
$\chi_n(B(a_0, \eps))$ contains $B(0, \delta)$. 
Then, $\{ \ell_{n,a}(K_n(a)) : a \in B(a_0, \eps)\}$ must contain $B(K_n(a_0), q_n \delta)$. 
In particular, if $n$ is large enough (to also make $q_n \delta \geq 4$), 
$\{ \ell_{n,a}(K_n(a)) : a \in B(a_0, \eps)\}$ contains a ball of radius $4$ around a point in the left half-plane $\re w < \log r_0$. 
	
Now we may employ Proposition~\ref{P:bands-in-petals}, with $r_0$ and $\Lambda = \overline{B(a_0, \eps)}$, 
to conclude that  for sufficiently large $n$, there are $a,b$ in $B(a_0, \eps)$ and $k_a$ and $k_b= k_a+1$ 
satisfying the desired properties in the proposition.}
\end{proof}

\begin{prop}\label{prop:unstable}
For every $a_0 \in \Omega$ and every neighbourhood $V \subset \C$ of $a_0$, there exists $N \in \N$ such that for all $n \geq N$, 
the family $\{ f_{\la_n,a} \}_{a \in V}$ is not $J$-stable. 
\end{prop}

\begin{proof}
{ Let us choose $\eps>0$ so that $B(a_0, \eps) \subset V \cap \Omega$. 
Then, we apply Proposition~\ref{P:perturb} to obtain the $N \in \N$. 
For each $n\geq N$, we obtain $a, b \in V$, and $k_a, k_b \in \mathbb{Z}$ satisfying the properties in that proposition.	 

Assume for a contradiction that the family $\{f_{\la_n,a} \}_{a \in V}$ is $J$-stable, for some $n\geq N$. 
Combining the above paragraph with Corollary~\ref{C:petal-in-basin}, we obtain 
\[f_{\la_n,a}^{\circ k}(c_{\la_n,a}) \in \mathcal{P}_{n,a}^{k_a}.\]
Then by $J$-stability, we must have 
\[f_{\la_n,s}^{\circ k} (c_{\la_n,s}) \in \mathcal{P}_{n,s}^{k_a}\]
for all $s \in V$. But this is not the case since 
\[f_{\la_n,b}^{\circ k}(c_{\la_n,b}) \in \mathcal{P}_{n,b}^{k_b}.\qedhere\]}
\end{proof}

\begin{rem}
{ By the seminal work of Yoccoz \cite{Yoc95}, the Brjuno condition is sharp (necessary and sufficient) 
for the linearisability of a quadratic polynomial with an irrationally indifferent fixed point. 
That optimality remains open for cubic polynomials in general.
However, it follows from a general result Perez-Marco that if $f_{e^{2i\pi\theta,a_0}}$ is a Siegel polynomial of capture type, 
for some $\theta \in \R \setminus \Q$ and $a_0 \in \C$, then $\theta$ must be a Brjuno number. 
That is because if $f_{e^{2i\pi\theta,a}}$ is of capture type, there is an open set $U \subset \C$ containing $a_0$ such that for all $a \in U$, 
$f_{e^{2i\pi\theta, a}}$ has a Siegel disk. 
But by \cite{Pe2001}, if $\theta \in \R \setminus \Q$ is non-Brjuno, then 
\[\{a \in \C: f_{e^{2i\pi\theta,a}} \text{ has a Siegel disk } \}\]
must be a polar set, which is impossible here.}

The same result of Perez-Marco \cite{Pe2001} implies that if $\theta$ is not Brjuno but $f_{e^{2i\pi\theta},a_0}$ has a Siegel disk containing $0$, 
then $a_0$ is in the bifurcation locus of the slice $\{f_{e^{2i\pi\theta,a}} , a \in \C\}$. 
Then by \cite{MSS83} or \cite{Ly83}, arbitrarily close to $a_0$, there must be some $a_1 \in \C$ such that $f_{e^{2i\pi\theta,a_1}}$ has a neutral cycle 
which is non-persistent in the family $\{f_{e^{2i\pi\theta,a}} , a \in \C\}$. 
In particular, $f_{e^{2i\pi\theta},a_1}$ has 2 neutral cycles, therefore by \cite{buff2009bifurcation} it is in the support of $\mubif$.
This proves that if there exists cubic polynomial maps with Siegel disks with non-Brjuno rotation numbers (which are conjectured not to exist), 
then they must also be in { $\mubif$}. 
\end{rem}

{ We have now completed the ingredients we need to prove Theorem~\ref{th:supp} for rotation numbers of bounded type.
In order to generalise it to all Brjuno numbers, we shall use the following result due to Avila, Buff and Chéritat.}  

\begin{theo}[\cite{ABC04}, Theorem 3, p.~11]\label{th:abc}
Let $\theta$ be a Brjuno number, and let $\Phi$ denote the Brjuno-Yoccoz function. Let $\theta_n \to \theta$, where $\theta_n$ are Brjuno numbers and 
\[\limsup_{n \to + \infty}\Phi(\theta_n) \to \Phi(\theta)+C\]
for some constant $C \geq 0$.
Let $f_n(z)=e^{2\pi i \theta_n} z+ O(z^2)$, for $n \in \N$, be a sequence of holomorphic maps on $\D$ such that 
$f_n \to R_\theta$ (the rigid rotation of angle $\theta$) locally uniformly on $\D$. 
Then 
\[\liminf_{n \to +\infty} \log r(\Delta_{f_n}) \geq e^{-C}.\]
\end{theo}
{ In the above statement, $r(\Delta_{f_n})$ denotes the conformal radius of the Siegel disk $\Delta_{f_n}$ of $f_n$ centred at $0$.} 

{ We now present the proof of the main result.}

\begin{proof}[Proof of Theorem \ref{th:supp}]
{ Let $a_0$ be an arbitrary parameter in $U$, which is not the centre of $U$. 
We consider two cases, based on whether $a_0$ lies in a non-degenerate parabolic locus or not.

\medskip

Case i) The parameter $a_0$ does not belong to a non-degenerate parabolic locus.  

\smallskip 

Let us first show that if $b_n(c)=0$ for some $n \in \N$ and $c \in \C$, then $(\la_n,c) \in  \supp \mubif$.
Indeed, by definition of $b_n$, the map $f_{\la_n,c}$ has a parabolic fixed point of parabolic multiplicity 2 at the origin. 
The immediate parabolic basin is therefore comprised of $2q_n$ Fatou components. 
On the other hand, when $b_n(a) \neq 0$, the parabolic multiplicity is $1$ and there are only $q_n$ Fatou components in the immediate parabolic basin of $0$. 
Therefore, the map $f_{\la_n,c}$ cannot be $J$-stable in the slice family $\{f_{\la_n,a} : a \in \C\}$. 
The classical results of  \cite{MSS83} and \cite{Ly83} imply the existence of $a_n'$ arbitrarily close to $c$ such that 
$f_{\la_n,a_n'}$ has a neutral cycle which is non-persistent in the slice $\{f_{\la_n,a} : a \in \C\}$.
Therefore $f_{\la_n,a_n'}$ has exactly 2 neutral cycles, including the parabolic fixed point at the origin. 
By \cite{buff2009bifurcation}, $(\la_n, a_n') \in \supp \mubif$, and since we may choose $a_n'$ arbitrarily close to $c$, so is $(\la_n,c)$.

By definition, if $a_0$ does not belong to any non-degenerate parabolic locus, there is a sequence of parameters $(a_k)_{k\geq 1}$ in $U$ 
such that $a_k \to a_0$, $b_{n_k}(a_k)=0$ for some $n_k \in \N$, and $\lim_{k \to \infty} n_k= \infty$. 
By the above paragraph, each $(\la_{n_k}, a_{k}) \in \supp \mubif$, and hence $(\la, a) \in \supp \mubif$. 
}	

\medskip 

{ Case ii) The parameter $a_0$ belongs to a non-degenerate parabolic locus, say $\Omega$.}

\smallskip 

Let us first assume that $\theta$ is bounded type. Fix an arbitrary $\eps>0$. 
We will find some $a_n \in \C$ such that $\| (\la_n,a_n) - (\la,a_0)\| \leq \eps$ and $(\la_n, a_n) \in \supp \mubif$.
	
Applying Proposition \ref{prop:unstable} with $V:=\D(a_0,\eps/2)$, we conclude that there exists $n$ 
such that $|\la_n - \la|<\eps/2$ and $a_n \in \D(a_0,\eps/2)$ such that $f_{\lam_n,a_n}$ has a neutral cycle besides the fixed point $0$. 
By \cite[Main Theorem]{buff2009bifurcation}, $(\la_n,a_n) \in \supp \mubif$, and by construction $\|(\la_n,a_n) - (\la,a_0)\| \leq \eps$.
Therefore $a_0 \in \supp \mubif$.
	
\medskip
	
Now let $\theta \in \R \setminus \Q$ be Brjuno number and $a \in \C$ be such that $f_{e^{2i\pi\theta},a}$ is of capture type 
(we do not assume anymore that $\theta$ has bounded type).
Let $[a_0: a_1: \ldots]$ denote the entries of its continued fraction expansion. 
For all $n \in \N$, let $\theta_n$ be the unique irrational number whose continued fraction expansion has entries 
\[[a_0: \ldots a_n: 1 : 1\ldots ].\]
We have
\[\limn \Phi(\theta_n) = \Phi(\theta)\]
where $\Phi$ is the Brjuno-Yoccoz function. 
Let $\phi: \Delta_{\theta}(a) \to \D$ denote the linearising coordinate of the Siegel disk of $f_{e^{2i\pi\theta},a}$, normalised to map 
$\Delta_{\theta}(a)$ to the unit disk $\D$. 
Let $h_n(z):=\phi \circ f_{e^{2i\pi\theta_n},a} \circ \phi^{-1}(z)$.
	
The sequence $h_n$ satisfies the assumptions of Theorem \ref{th:abc} with $C=0$. 
It follows that for every compact $K \subset \Delta_\theta(a)$, there exists $n_0 \in \N$ such that for all $n \geq n_0$, $K \subset \Delta_{\theta_n}(a)$. 
By assumption, there exists $k \in \N^*$ such that $f_{e^{2i\pi\theta},a}^{\circ k}(c_a) \in \Delta_\theta(a)$, where $c_a$ is a critical point of 
$f_{e^{2i\pi\theta},a}$; by taking $K:=\overline{\D}(c_a,\delta)$ for $\delta>0$ small enough that $K \subset \Delta_\theta(a)$, 
we deduce that for all $n$ large enough, 	$f_{e^{2i\pi\theta_n},a}^{\circ k}(c_a) \in \Delta_{\theta_n}(a)$. 
In other words, for all $n$ large enough, $f_{e^{2i\pi\theta_n},a}$ is a Siegel polynomial of capture type, with bounded type rotation number. 
By the above, $(e^{2i\pi\theta_n},a) \in \supp \mubif$, and since $e^{2i\pi \theta_n} \to e^{2i\pi \theta}$, $(e^{2i\pi\theta},a)$  must be in $\supp \mubif$ as well.
\end{proof}

\bibliographystyle{amsalpha}
\bibliography{Data}
\end{document}